\documentclass[12pt,american]{article}
\usepackage[T1]{fontenc}
\usepackage[latin9]{inputenc}
\usepackage{babel}
\usepackage{amsmath}
\usepackage{amsthm}
\usepackage{amssymb}
\usepackage[unicode=true,bookmarks=true,bookmarksnumbered=false,bookmarksopen=false,breaklinks=true,pdfborder={0 0 0},backref=false,colorlinks=false]
 {hyperref}
\hypersetup{pdftitle={Green Measures for a Class of non-Markov Processes},
 pdfauthor={H. P. Suryawan, and J. L. da Silva},
 pdfsubject={Green Measures for a Class of non-Markov Processes},
 pdfkeywords={Fractional Brownian motion; generalized grey Brownian motion; Green measure; subordination}}

\makeatletter

\theoremstyle{plain}
\newtheorem{thm}{\protect\theoremname}[section]
\theoremstyle{definition}

\theoremstyle{definition}
\newtheorem{defn}[thm]{\protect\definitionname}
\theoremstyle{plain}
\newtheorem{cor}[thm]{\protect\corollaryname}
\theoremstyle{remark}
\newtheorem{rem}[thm]{\protect\remarkname}
\newtheorem{prop}[thm]{\protect\propositionname}

\makeatother

\providecommand{\corollaryname}{Corollary}
\providecommand{\definitionname}{Definition}
\providecommand{\examplename}{Example}
\providecommand{\remarkname}{Remark}
\providecommand{\theoremname}{Theorem}
\providecommand{\propositionname}{Proposition}

\begin{document}
\title{Green Measures for a Class of non-Markov Processes}
\author{\textbf{Herry Pribawanto Suryawan}\\
 Department of Mathematics, Sanata Dharma University\\
 55281 Yogyakarta, Indonesia\\
 Email: herrypribs@usd.ac.id \and \textbf{Jos{\'e} Lu{\'i}s
da Silva}\\
 CIMA, University of Madeira, Campus da Penteada,\\
 9020-105 Funchal, Portugal\\
 Email: joses@staff.uma.pt}
\date{\today}

\maketitle
\begin{abstract}
In this paper, we investigate the Green measure for a class of non-Gaussian
processes in $\mathbb{R}^{d}$. These measures are associated with
the family of generalized grey Brownian motions $B_{\beta,\alpha}$,
$0<\beta\le1$, $0<\alpha\le2$. This family includes both fractional
Brownian motion, Brownian motion, and other non-Gaussian processes.
We show that the perpetual integral exists with probability $1$ for
$d\alpha>2$ and $1<\alpha\le2$. The Green measure then generalizes
those measures of all these classes.
\\[.3cm]
\textbf{Keywords}: Fractional Brownian motion; generalized grey Brownian motion; Green measure; subordination.
\end{abstract}

\tableofcontents{}

\section{Introduction}

The goal of this paper (see Theorem~\ref{thm:Green-meas-ggBm}
and Corollary~\ref{cor:GM-ggBm} below) is to prove the existence
of the Green measure for a class of non-Gaussian processes in $\mathbb{R}^{d}$,
called generalized grey Brownian motion (ggBm for short). We denote
this family of processes by $B_{\beta,\alpha}$ with parameters $0<\beta\le1$ and $0<\alpha\le2$.
More precisely, for a Borel function $f:\mathbb{R}^{d}\longrightarrow\mathbb{R}$,
the potential of $f$ (see \cite{Blumenthal1968,Revuz-Yor-94} for
details) is defined as 
\begin{equation}
V_{\beta,\alpha}(f,x)=\int_{0}^{\infty}\mathbb{E}\big[f(x+B_{\beta,\alpha}(t))\big]\,\mathrm{d}t,\quad x\in\mathbb{R}^{d}.\label{eq:potential_of_f}
\end{equation}
We would like to investigate the class of functions $f$ for which
the potential of $f$ has the representation
\begin{equation}
V_{\beta,\alpha}(f,x)=\int_{\mathbb{R}^{d}}f(y)\mathcal{G}_{\beta,\alpha}(x,\mathrm{d}y),\label{eq:green-measure}
\end{equation}
where $\mathcal{G}(x,\cdot):=\mathcal{G}_{\beta,\alpha}(x,\cdot)$
is a Radon measure on $\mathbb{R}^{d}$ called Green measure corresponding
to the ggBm $B_{\beta,\alpha}$, see Definition~\ref{def:SP-GM} below.
First, we establish the existence of the perpetual integral (cf.~Theorem
\ref{thm:Green-meas-ggBm})
\[
\int_{0}^{\infty}f(x+B_{\beta,\alpha}(t))\,\mathrm{d}t
\]
with probability $1$. This leads to an explicit representation
of the Green measure for ggBm, namely (cf.~Corollary~\ref{cor:GM-ggBm})
\[
\mathcal{G}_{\beta,\alpha}(x,\mathrm{d}y)=\frac{D}{|x-y|^{d-2/\alpha}}\,\mathrm{d}y,\quad d\alpha>2,\quad 1<\alpha\le2,
\]
where $D$ is a constant that depends on $\beta,\alpha$, and the
dimension $d$; see Theorem~\ref{thm:Green-meas-ggBm} for the explicit expression. Note that
as $d\alpha>2$ and $1<\alpha\le2$, the Green measure $\mathcal{G}_{\beta,\alpha}(x,\cdot)$ exists for $d\ge2$, since $d>2/\alpha\in[1,2)$. The Brownian case ($\alpha=1)$ is covered only for $d\ge3$. 

We emphasize that the existence of the Green measure
for a given process $X$ is not always guaranteed. As an example,
the $d$-dimensional Brownian motion (Bm) starting at $x\in\mathbb{R}^{d}$ has a density given by $p_{t}(x,y)=(2\pi t)^{-d/2}\exp\big(-|x-y|^{2}/(2t)\big)$,
$y\in\mathbb{R}^{d}$. It is not difficult to see that $\int_{0}^{\infty}p_{t}(x,y)\,\mathrm{d}t$
does not exist for $d=1,2$. This implies the non-existence of the
Green measure of Bm for $d=1,2$. On the other hand, for $d\ge3$, the Green measure of Bm on $\mathbb{R}^{d}$ exists and is given by
$\mathcal{G}(x,\mathrm{d}y)=C(d)|x-y|^{2-d}\,\mathrm{d}y$, where $C(d)$ is a constant depending on
the dimension $d$; see \cite{KMS2020} and the references therein for more details. In a two-dimensional space, the Green measure of ggBm is determined by the parameter $\alpha$. The Green measure of ggBm for $d=1$ requires further analysis (for Bm see \cite{Armitage2001}, Ch.~4) which we will postpone for a future paper.

The paper is organized as follows. In Section~\ref{sec:ggBm} we
recall the definition and main properties of ggBm that will be needed
later. In Section~\ref{sec:GM-ggBm} we show the existence of the
perpetual integral with probability $1$, which leads to the explicit
formula for the Green measure for ggBm.  In Section~\ref{sec:Discussion-and-Conclusions}, we discuss the results obtained, connect them with other topics, and draw conclusions.

\section{Generalized Grey Brownian Motion}

\label{sec:ggBm}We recall the class of non-Gaussian processes, called the
generalized grey Brownian motion, which we study below. This class
of processes was first introduced by Schneider \cite{Schneider90a,Schneider90},
and was generalized by Mura et al.~(see \cite{Mura_Pagnini_08,Mura_mainardi_09})
as a stochastic model for slow/fast anomalous diffusion described
by the time fractional diffusion equation.

\subsection{Definition and Properties}

For $0<\beta\le1$ the Mittag-Leffler (entire) function $E_{\beta}$
is defined by the Taylor series 
\begin{equation}
E_{\beta}(z):=\sum_{n=0}^{\infty}\frac{z^{n}}{\Gamma(\beta n+1)},\quad z\in\mathbb{C},\label{MLF}
\end{equation}
where 
\[
\Gamma(z)=\int_{0}^{\infty}t^{z-1}e^{-t}\,dt,\quad z\in\mathbb{C},\,\mathrm{Re}(z)\geq0
\]
is the Euler gamma function.

\noindent The $M$-Wright function is a special case of the class
of Wright functions $W_{\lambda,\mu}$, $\lambda>-1$, $\mu\in\mathbb{C}$
via 
\[
M_{\beta}(z):=W_{-\beta,1-\beta}(-z)=\sum_{n=0}^{\infty}\frac{(-z)^{n}}{n!\Gamma(-\beta n+1-\beta)}.
\]
The special choice $\beta=1/2$ yields the Gaussian density on $[0,\infty)$
\begin{equation}
M_{\frac{1}{2}}(z)=\frac{1}{\sqrt{\pi}}\mathrm{exp}\left(-\frac{z^{2}}{4}\right).\label{eq:MWright_Gaussian}
\end{equation}

\noindent The Mittag-Leffler function $E_{\beta}$ is the Laplace
transform of the $M$-Wright function, that is, 
\begin{equation}
E_{\beta}(-s)=\int_{0}^{\infty}e^{-s\tau}M_{\beta}(\tau)\,d\tau.\label{eq:M_wright}
\end{equation}
The generalized moments of the density $M_{\beta}$ of order $\delta>-1$
are finite and are given (see \cite{Mura_Pagnini_08}) by 
\begin{equation}
\int_{0}^{\infty}\tau^{\delta}M_{\beta}(\tau)\,d\tau=\frac{\Gamma(\delta+1)}{\Gamma(\beta\delta+1)}.\label{eq:Mbeta-moments}
\end{equation}

\begin{defn}
Let $0<\beta\le1$ and $0<\alpha\le2$ be given. A $d$-dimensional
continuous stochastic process $B_{\beta,\alpha}=\big\{ B_{\beta,\alpha}(t),\,t\geq0\big\}$
starting at $0\in\mathbb{R}^{d}$ and defined on a complete probability
space $(\Omega,\mathcal{F},\mathbb{P})$, is a ggBm in $\mathbb{R}^{d}$
(see \cite{Mura_mainardi_09} for $d=1$) if: 
\begin{enumerate}
\item $\mathbb{P}(B_{\beta,\alpha}(0)=0)=1$, that is, $B_{\beta,\alpha}$
starts at zero $\mathbb{P}$-almost surely ($\mathbb{P}$-a.s.).
\item Any collection $\big\{ B_{\beta,\alpha}(t_{1}),\ldots,B_{\beta,\alpha}(t_{n})\big\}$
with $0\leq t_{1}<t_{2}<\dots<t_{n}<\infty$ has a characteristic function
given, for any $\theta=(\theta_{1},\ldots,\theta_{n})\in(\mathbb{R}^{d})^{n}$
with $\theta_{k}=(\theta_{k,1},\dots\theta_{k,d})$, $k=1,\dots,n$, by 
\begin{equation}
\mathbb{E}\left[\mathrm{exp}\left(\mathrm{i}\sum_{k=1}^{n}(\theta_{k},B_{\beta,\alpha}(t_{k}))_{\mathbb{R}^d}\right)\right]=E_{\beta}\left[-\frac{1}{2}\sum_{j=1}^{d}(\theta_{.,j},\gamma_{\alpha}\theta_{.,j})_{\mathbb{R}^{n}}\right],\label{eq:charact-func-ggBm}
\end{equation}
where $\mathbb{E}$ denotes the expectation w.r.t.~$\mathbb{P}$
and
\[
\gamma_{\alpha}:=\gamma_{\alpha,n}:=\big(t_{k}^{\alpha}+t_{j}^{\alpha}-|t_{k}-t_{j}|^{\alpha}\big)_{k,j=1}^{n}.
\]
\item The joint probability density function of $\big(B_{\beta,\alpha}(t_{1}),\ldots,B_{\beta,\alpha}(t_{n})\big)$
is equal to
\begin{equation}
\rho_{\beta}(\theta,\gamma_{\alpha})=\frac{(2\pi)^{-\frac{nd}{2}}}{(\mathrm{det}\gamma_{\alpha})^{d/2}}\int_{0}^{\infty}\tau^{-\frac{nd}{2}}\mathrm{e}^{-\frac{1}{2\tau}\sum_{j=1}^{d}(\theta_{\cdot,j},\gamma_{\alpha}^{-1}\theta_{\cdot,j})_{\mathbb{R}^{n}}}M_{\beta}(\tau)\,d\tau.\label{eq:fdd-density-ggBm}
\end{equation}
\end{enumerate}
\end{defn}

The following are the most important key properties of ggBm:
\begin{enumerate}
\item[(P1). ] For each $t\geq0$, the moments of any order of $B_{\beta,\alpha}(t)$ are given by 
\[
\begin{cases}
\mathbb{E}\big[|B_{\beta,\alpha}(t)|^{2n+1}\big] & =0,\\
\noalign{\vskip4pt}\mathbb{E}\big[|B_{\beta,\alpha}(t)|^{2n}\big] & =\frac{(2n)!}{2^{n}\Gamma(\beta n+1)}t^{\alpha n}.
\end{cases}
\]
\item[(P2).] The covariance function has the form 
\begin{equation}
\mathbb{E}\big[\big(B_{\beta,\alpha}(t),B_{\beta,\alpha}(s)\big)\big]=\frac{d}{2\Gamma(\beta+1)}\big(t^{\alpha}+s^{\alpha}-|t-s|^{\alpha}\big),\quad t,s\geq0.\label{eq:auto-cv-gBm}
\end{equation}
\item[(P3).] For each $t,s\geq0$, the characteristic function of the increments
is 
\begin{equation}
\mathbb{E}\big[e^{i(k,B_{\beta,\alpha}(t)-B_{\beta,\alpha}(s))}\big]=E_{\beta}\left(-\frac{|k|^{2}}{2}|t-s|^{\alpha}\right),\quad k\in\mathbb{R}^{d}.\label{eq:cf_gBm_increments}
\end{equation}
\item[(P4).] The process $B_{\beta,\alpha}$ is non-Gaussian, $\alpha/2$-self-similar
with stationary increments.
\item[(P5).] The ggBm is not a semimartingale. Furthermore, $B_{\alpha,\beta}$
cannot be of finite variation in $[0,1]$ and, by scaling and stationarity
of the increment, on any interval in $\mathbb{R}^{+}$. 
\item[(P5).] For $n=1$, the density $\rho_{\beta}(x,t)$, $x\in\mathbb{R}^{d}$,
$t>0$, is the fundamental solution of the following fractional
differential equation (see \cite{Mentrelli2015}) 
\[
\mathbb{D}_{t}^{2\beta}\rho_{\beta}(x,t)=\Delta_{x}\rho_{\beta}(x,t),
\]
where $\Delta_{x}$ is the $d$-dimensional Laplacian in $x$ and
$\mathbb{D}_{t}^{2\beta}$ is the Caputo-Dzherbashian fractional derivative;
see \cite{SKM1993} for the definition and properties.
\end{enumerate}

\subsection{Representations of Generalized Grey Brownian Motion}

The ggBm admits different representations in terms of well-known processes.
It follows from \eqref{eq:charact-func-ggBm} that ggBm has an elliptical
distribution, see Section~2 in \cite{DaSilva2014} or Section~3
in \cite{GJ15}. On the other hand, ggBm is also given as a product
(see \cite{Mura_Pagnini_08} for $d=1$) of two processes as follows
\begin{equation}
\big\{ B_{\beta,\alpha}(t),\,t\geq0\big\}\overset{\mathcal{L}}{=}\big\{\sqrt{Y_{\beta}}B^{\alpha/2}(t),\,t\geq0\big\}.\label{eq:repr-Mura}
\end{equation}
Here, $\overset{\mathcal{L}}{=}$ means equality in law, the nonnegative
random variable $Y_{\beta}$ has density $M_{\beta}$ and $B^{\alpha/2}$
is a $d$-dimensional fBm with Hurst parameter $\alpha/2$ and independent
of $Y_{\beta}$. 

We give another representation of ggBm $B_{\beta,\alpha}$ as
a subordination of fBm (see Prop.~2.14 in \cite{ERdS2023} for $d=1$)
which is used below. For completeness, we give the short proof.
\begin{prop}
\label{prop:subord-ggBm}The ggBm has the following representation
\begin{equation}
\big\{ B_{\beta,\alpha}(t),\,t\geq0\big\}\overset{\mathcal{L}}{=}\big\{ B^{\alpha/2}(tY_{\beta}^{1/\alpha}),\,t\geq0\big\}.\label{eq:repr-our}
\end{equation}
\end{prop}

\begin{proof}
We must show that both representations \eqref{eq:repr-Mura} and
\eqref{eq:repr-our} have the same finite-dimensional distribution.
For every $\theta=(\theta_{1},\ldots,\theta_{n})\in(\mathbb{R}^{d})^{n}$,
we have 
\begin{align*}
&\mathbb{E}\left[\exp\left(\mathrm{i}\sum_{k=1}^{n}\big(\theta_{k},B^{\alpha/2}(t_{k}Y_{\beta}^{1/\alpha})\big)\right)\right]\\ & =  \int_{0}^{\infty}\mathbb{E}\left[\exp\left(\mathrm{i}\sum_{k=1}^{n}\big(\theta_{k},B^{\alpha/2}(t_{k}y^{1/\alpha})\big)\right)\right]M_{\beta}(y)\,\mathrm{d}y\\
 & =  \int_{0}^{\infty}\mathbb{E}\left[\exp\left(\mathrm{i}\sum_{k=1}^{n}\big(\theta_{k},y^{1/2}B^{\alpha/2}(t_{k})\big)\right)\right]M_{\beta}(y)\,\mathrm{d}y\\
 & =  \mathbb{E}\left[\exp\left(\mathrm{i}\sum_{k=1}^{n}\big(\theta_{k},Y_{\beta}^{1/2}B^{\alpha/2}(t_{k})\big)\right)\right].
\end{align*}
In the second equality, we used the $\alpha/2$-self-similarity of
fBm. This completes the proof.
\end{proof}

%We use the representation \eqref{eq:repr-our} to compute the density
%of $B_{\beta,\alpha}(t)$, $t\ge0$. Let us compute the density. 

\section{The Green Measure for Generalized Grey Brownian Motion}
\label{sec:GM-ggBm}In this section we show the existence of the Green
measure for the ggBm, see \eqref{eq:potential_of_f} and \eqref{eq:green-measure}.
Let us begin by discussing the existence of the Green measure for a
general stochastic process $X$. 

Let $X=\{X(t),\;t\ge0\}$ be a stochastic process in $\mathbb{R}^{d}$
starting from $x\in\mathbb{R}^{d}$. If $X(t)$, $t\ge0$, has a probability
distribution $\rho_{X(t)}(x,\cdot)$, then Eq.~\eqref{eq:potential_of_f}
becomes
\begin{equation}
V_{X}(x,f)=\int_{0}^{\infty}\int_{\mathbb{R}^{d}}f(y)\,\rho_{X(t)}(x,\mathrm{d}y)\,\mathrm{d}t.\label{eq:potential-f}
\end{equation}
Then, applying the Fubini theorem, the Green measure $\mathcal{G}_{X}(x,\cdot)$
of $X$ is given by 
\[
\mathcal{G}_{X}(x,\mathrm{d}y)=\int_{0}^{\infty}\rho_{X(t)}(x,\mathrm{d}y)\,\mathrm{d}t,
\]
assuming the existence of $\mathcal{G}_{X}(x,\cdot)$ as a Radon measure
on $\mathbb{R}^{d}$. That is, for every bounded Borel set $B\in\mathcal{B}_{b}(\mathbb{R}^{d})$
we have 
\[
\mathcal{G}_{X}(x,B)=\int_{0}^{\infty}\rho_{X(t)}(x,B)\,\mathrm{d}t<\infty.
\] 
If the probability distribution $\rho_{X(t)}(x,\cdot)$
is also absolutely continuous with respect to the Lebesgue measure, say
$\rho_{X(t)}(x,\mathrm{d}y)=\rho_{t}(x,y)\,\mathrm{d}y$, then the
function 
\begin{equation}
g_{X}(x,y):=\int_{0}^{\infty}\rho_{t}(x,y)\,\mathrm{d}t,\quad\forall y\in\mathbb{R}^{d},\label{eq:Green_function}
\end{equation}
is called the Green function of the stochastic process $X$. Moreover,
the Green measure in this case is given by $\mathcal{G}_{X}(x,\mathrm{d}y)=g_{X}(x,y)\,\mathrm{d}y$. 

This leads us to the following definition of the Green measure of a stochastic
process $X.$
\begin{defn}
\label{def:SP-GM} Let $X=\{X(t),\;t\ge0\}$ be a stochastic process
on $\mathbb{R}^{d}$ starting from $x\in\mathbb{R}^{d}$ and $\rho_{X(t)}(x,\cdot)$
be the probability distribution of $X(t)$, $t\ge0$. The Green measure
of $X$ is defined as a Radon measure on $\mathbb{R}^{d}$ by 
\[
\mathcal{G}_{X}(x,B):=\int_{0}^{\infty}\rho_{X(t)}(x,B)\,\mathrm{d}t,\;\;B\in\mathcal{B}_{b}(\mathbb{R}^{d}),
\]
or 
\[
\int_{\mathbb{R}^{d}}f(y)\mathcal{G}_{X}(x,\mathrm{d}y)=\int_{\mathbb{R}^{d}}f(y)\int_{0}^{\infty}\rho_{X(t)}(x,dy)\,\mathrm{d}t,\;\;f\in C_{0}(\mathbb{R}^{d})
\]
whenever these integrals exist. 
\end{defn}

In other words, $\mathcal{G}_{X}(x,B)$ is the expected length of
time the process remains in $B$. 

In order to state the main theorem which establishes the existence
of the Green measure for ggBm, first, we introduce a proper Banach
space of functions $f:\mathbb{R}^{d}\longrightarrow\mathbb{R}$ such
that
\[
\int_{0}^{\infty}f\big(x+B_{\beta,\alpha}(t)\big)\,\mathrm{d}t
\]
is finite $\mathbb{P}$-a.s. Without loss of generality, we may assume that
$f\ge0$ above. We define the space $CL(\mathbb{R}^{d})$ of continuous real-valued on $\mathbb{R}^{d}$ by

\[
CL(\mathbb{R}^{d}):=\big\{ f:\mathbb{R}^{d}\rightarrow\mathbb{R}\mid f\;\text{is continuous, bounded and}\;f\in L^{1}(\mathbb{R}^{d})\big\}.
\]
The space $CL(\mathbb{R}^{d})$ becomes a Banach space with the norm
\[
\|f\|_{CL}:=\|f\|_{\infty}+\|f\|_{1},\quad\forall f\in CL(\mathbb{R}^{d}),
\]
where $\|\cdot\|_{\infty}$ denotes the $\sup$-norm and $\|\cdot\|_{1}$
is the norm in $L^{1}(\mathbb{R}^{d}$). The choice of $CL(\mathbb{R}^{d})$
allows us to show that the family of random variables (also known as
perpetual integral functionals)
\[
\int_{0}^{\infty}f\big(x+B_{\beta,\alpha}(t)\big)\,\mathrm{d}t,\quad f\in CL(\mathbb{R}^{d})
\]
have finite expectations $\mathbb{P}$-a.s.
\begin{thm}
\label{thm:Green-meas-ggBm}Let $f\in CL(\mathbb{R}^{d})$ and $x\in\mathbb{R}^{d}$
be given and consider ggBm $B_{\beta,\alpha}$ with $d\alpha>2$
and $1<\alpha\le2$.
Then, the perpetual integral functional $\int_{0}^{\infty}f(x+B(t))\,\mathrm{d}t$
is finite $\mathbb{P}$-a.s.~and its expectation equals 
\begin{equation}
\begin{gathered}\mathbb{E}\left[\int_{0}^{\infty}f(x+B_{\beta,\alpha}(t))\,\mathrm{d}t\right]=D\int_{\mathbb{R}^{d}}\frac{f(x+y)}{|y|^{d-2/\alpha}}\,\mathrm{d}y,\end{gathered}
\label{eq:fract}
\end{equation}
where $D=D(\beta,\alpha,d)=\frac{1}{\alpha}2^{-1/\alpha}\pi^{-\frac{d}{2}}\Gamma\left(\frac{d}{2}-\frac{1}{\alpha}\right)\frac{\Gamma(1-\frac{1}{\alpha})}{\Gamma(1-\frac{\beta}{\alpha})}$. 
\end{thm}

\begin{proof}
Given $x\in\mathbb{R}^{d}$ and $f\in CL(\mathbb{R}^{d})$ non-negative,
let $\rho_{\beta}(\cdot,t^\alpha)$ denote the density of $B_{\beta,\alpha}(t)$,
$t\ge0$, given by (see \eqref{eq:fdd-density-ggBm} with $n=1$)
\[
\rho_{\beta}(y,t^\alpha)=\frac{1}{(2\pi t^{\alpha})^{d/2}}\int_{0}^{\infty}\tau^{-d/2}\mathrm{e}^{-\frac{|y|^{2}}{2t^{\alpha}\tau}}M_{\beta}(\tau)\,\mathrm{d}\tau,\quad y\in\mathbb{R}^d.
\]
First, we show the equality \eqref{eq:fract}. It follows from the above considerations that 
\begin{multline*}
\mathbb{E}\left[\int_{0}^{\infty}f(x+B_{\beta,\alpha}(t))\,\mathrm{d}t\right]=\int_{0}^{\infty}\int_{\mathbb{R}^{d}}f(x+y)\rho_{t}^{\beta,\alpha}(y)\,\mathrm{d}y\,\mathrm{d}t.\\
=\int_{0}^{\infty}\int_{\mathbb{R}^{d}}f(x+y)\frac{1}{(2\pi t^{\alpha})^{d/2}}\int_{0}^{\infty}\tau^{-d/2}M_{\beta}(\tau)e^{-\frac{|y|^{2}}{2t^{\alpha}\tau}}\,\mathrm{d}\tau\,\mathrm{d}y\,\mathrm{d}t.
\end{multline*}
Using Fubini's Theorem, we first compute the $t$-integral and use
the assumption $d\alpha>2$. We obtain
\[
\int_{0}^{\infty}\frac{1}{(2\pi t^{\alpha}\tau)^{d/2}}e^{-\frac{|y|^{2}}{2t^{\alpha}\tau}}\,\mathrm{d}t=C(\alpha,d)\frac{\tau^{-\frac{1}{\alpha}}}{|y|^{d-2/\alpha}},
\]
where 
\[
C(\alpha,d):=\frac{1}{\alpha}2^{-1/\alpha}\pi^{-\frac{d}{2}}\Gamma\left(\frac{d}{2}-\frac{1}{\alpha}\right).
\]
Next we compute the $\tau$-integral using \eqref{eq:Mbeta-moments}
so that 
\[
\int_{0}^{\infty}\tau^{-1/\alpha}M_{\beta}(\tau)\,\mathrm{d}\tau=\frac{\Gamma(1-\frac{1}{\alpha})}{\Gamma(1-\frac{\beta}{\alpha})},\quad\alpha>1.
\] 
Combining gives
\[
\mathbb{E}\left[\int_{0}^{\infty}f(x+B_{\beta,\alpha}(t))\,\mathrm{d}t\right]=D\int_{\mathbb{R}^{d}}\frac{f(x+y)}{|y|^{d-2/\alpha}}\,\mathrm{d}y,
\]
where \[ D=D(\beta,\alpha,d)=C(\alpha,d)\frac{\Gamma(1-\frac{1}{\alpha})}{\Gamma(1-\frac{\beta}{\alpha})}.\]
Therefore, the equality \eqref{eq:fract} is shown.\\
Now we show that the right-hand side of \eqref{eq:fract} is finite
for every non-negative $f\in CL(\mathbb{R}^{d})$. To see this, we
may use the local integrability of $|y|^{d-2/\alpha}$ in $y$ and
obtain
\begin{eqnarray*}
\int_{\mathbb{R}^{d}}\frac{f(x+y)}{|y|^{d-2/\alpha}}\,\mathrm{d}y & =& \int_{\{|y|\le1\}}\frac{f(x+y)}{|y|^{d-2/\alpha}}\,\mathrm{d}y+\int_{\{|y|>1\}}\frac{f(x+y)}{|y|^{d-2/\alpha}}\,\mathrm{d}y\\
 & \le & C_{1}\|f\|_{\infty}+C_{2}\|f\|_{1}\leq C\|f\|_{CL}.
\end{eqnarray*}
 Therefore, the integral in \eqref{eq:fract} is, in fact, well defined. In
other words, the integral $\int_{0}^{\infty}f\big(x+B_{\beta,\alpha}(t)\big)\,\mathrm{d}t$
exists with probability 1. This completes the proof.
\end{proof}
As a consequence of the above theorem, we immediately obtain the Green
measure of ggBm $B_{\beta,\alpha}$, that is, comparing \eqref{eq:green-measure}
and \eqref{eq:fract}.
\begin{cor}
\label{cor:GM-ggBm}The Green measure of ggBm $B_{\beta,\alpha}$
for $d\alpha>2$ is given by
\[
\mathcal{G}_{\beta,\alpha}(x,\mathrm{d}y)=\frac{D}{|x-y|^{d-2/\alpha}}\,\mathrm{d}y.
\]
\end{cor}

\begin{rem}
\label{rem:green-measure}
\begin{enumerate}
\item It is possible to show that given $f\neq0$, the perpetual integral
$\int_{0}^{\infty}f(x+B_{\beta,\alpha}(t))\,\mathrm{d}t$ is a non-constant random
variable. As a consequence, for $f\ge0$ the variance of $\int_{0}^{\infty}f(x+B_{\beta,\alpha}(t))\,\mathrm{d}t$
is strictly positive. The proof uses the notion of conditional full
support of ggBm. We will not provide a detailed explanation of this result that closely follows the ideas of Theorem~2.2 in \cite{KMS2020} to which we address the interested readers.
\item Note also that the functional in \eqref{eq:potential_of_f} 
\[
V_{\beta,\alpha}(\cdot,x):CL(\mathbb{R}^{d})\longrightarrow\mathbb{R}
\]
is continuous. In fact, from the proof of Theorem~\ref{thm:Green-meas-ggBm}
 for any $f\in CL(\mathbb{R}^{d})$ yields
\[
|V_{\beta,\alpha}(f,x)|\le K\|f\|_{CL},
\]
where $K$ is a constant depending on the parameters $\beta,\alpha$,
and $d$. 
\end{enumerate}
\end{rem}

\section{Discussion and Conclusions}

\label{sec:Discussion-and-Conclusions}
We have derived the Green measure for the class of stochastic processes called the generalized grey Brownian motion in Euclidean space $\mathbb{R}^d$ for $d\ge2$. This class includes, in particular, fractional Brownian motion and other non-Gaussian processes. To address the case where $d=1$, a renormalization process is needed. However, this will be postponed to future work. For $\beta=\alpha=1$ ggBm $B_{1,1}$ is nothing but a Brownian motion. In this case, the Green measure exists for $d\ge3$. 

The relationship between the Green measure and the local time of the ggBm can be described as follows. For any $T>0$ and a continuous function $f:\mathbb{R}^d\longrightarrow\mathbb{R}$, the integral functional
\begin{equation}\label{eq:funct-integral-T}
\int_0^Tf(B_{\beta,\alpha}(t))\,\mathrm{d}t    
\end{equation}
is well-defined. For $d=1$ the integral \eqref{eq:funct-integral-T} with $f\in L^1(\mathbb{R})$ is represented as
\[
\int_0^Tf(B_{\beta,\alpha}(t))\,\mathrm{d}t =\int_\mathbb{R}f(x)L_{\beta,\alpha}(T,x)\,\mathrm{d}x,
\]
where $L_{\beta,\alpha}(T,x)$ is the local time of ggBm up to time $T$ at the point $x$, see \cite{DaSilva2014,GJ15}.  The Green measure corresponds to the asymptotic behaviour in $T$ of the expectation of local time $L_{\beta,\alpha}(T,x)$.  The existence of this asymptotic depends on the dimension $d$ and the transient or recurrent properties of the process. 

\subsection*{Acknowledgments}

This research was funded by FCT-Funda{\c c\~a}o para a Ci{\^e}ncia e a Tecnologia, Portugal grant number UIDB/MAT/04674/2020,  \url{https://doi.org/10.54499/UIDB/04674/2020} through the Center for Research in Mathematics and Applications (CIMA) related to the Statistics, Stochastic Processes and Applications (SSPA) group.


\begin{thebibliography}{KMdS21}

\bibitem[AG01]{Armitage2001}
D.~H. Armitage and S.~J. Gardiner.
\newblock {\em {Classical Potential Theory}}.
\newblock Springer Monographs in Mathematics. Springer, 2001.

\bibitem[BG68]{Blumenthal1968}
R.~M. Blumenthal and R.~K. Getoor.
\newblock {\em {Markov Processes and Potential Theory}}.
\newblock Academic Press, 1968.

\bibitem[dSE15]{DaSilva2014}
J.~L. da~Silva and M.~Erraoui.
\newblock Generalized grey {B}rownian motion local time: {E}xistence and weak approximation.
\newblock {\em Stochastics}, 87(2):347--361, October 2015.

\bibitem[ERdS24]{ERdS2023}
M.~Erraoui, M.~R{\"o}ckner, and J.~L. da~Silva.
\newblock Cameron-{M}artin type theorem for a class of non-{G}aussian measures.
\newblock Submitted, 2024.

\bibitem[GJ16]{GJ15}
M.~Grothaus and F.~Jahnert.
\newblock {Mittag-Leffler Analysis II: Application to the fractional heat equation}.
\newblock {\em J.\ Funct.\ Anal.}, 270(7):2732--2768, April 2016.

\bibitem[KMdS21]{KMS2020}
Y.~Kondratiev, Y.~Mishura, and J.~L. da~Silva.
\newblock Perpetual integral functionals of multidimensional stochastic processes.
\newblock {\em Stochastics}, 93(8):1249--1260, 2021.

\bibitem[MM09]{Mura_mainardi_09}
A.~Mura and F.~Mainardi.
\newblock A class of self-similar stochastic processes with stationary increments to model anomalous diffusion in physics.
\newblock {\em Integral Transforms Spec.\ Funct.}, 20(3-4):185--198, 2009.

\bibitem[MP08]{Mura_Pagnini_08}
A.~Mura and G.~Pagnini.
\newblock Characterizations and simulations of a class of stochastic processes to model anomalous diffusion.
\newblock {\em J.\ Phys.\ A: Math.\ Theor.}, 41(28):285003, 22, 2008.

\bibitem[MP15]{Mentrelli2015}
A.~Mentrelli and G.~Pagnini.
\newblock {Front propagation in anomalous diffusive media governed by time-fractional diffusion}.
\newblock {\em J.\ Comput.\ Phys.}, 293:427--441, 2015.

\bibitem[RY99]{Revuz-Yor-94}
D.~Revuz and M.~Yor.
\newblock {\em Continuous martingales and {B}rownian motion}, volume 293 of {\em Grundlehren der Mathematischen Wissenschaften [Fundamental Principles of Mathematical Sciences]}.
\newblock Springer-Verlag, Berlin, 3rd edition, 1999.

\bibitem[Sch90a]{Schneider90a}
W.~R. Schneider.
\newblock Fractional diffusion.
\newblock In R.~Lima, L.~Streit, and R.~Vilela~Mendes, editors, {\em Dynamics and stochastic processes ({L}isbon, 1988)}, volume 355 of {\em Lecture Notes in Phys.}, pages 276--286. Springer, New York, 1990.

\bibitem[Sch90b]{Schneider90}
W.~R. Schneider.
\newblock Grey noise.
\newblock In S.~Albeverio, G.~Casati, U.~Cattaneo, D.~Merlini, and R.~Moresi, editors, {\em {Stochastic Processes, Physics and Geometry}}, pages 676--681. World Scientific Publishing, Teaneck, NJ, 1990.

\bibitem[SKM93]{SKM1993}
S.~G. Samko, A.~A. Kilbas, and O.~I. Marichev.
\newblock {\em Fractional integrals and derivatives}.
\newblock Gordon and Breach Science Publishers, Yverdon, 1993.
\newblock Theory and applications, Edited and with a foreword by S.\ M.\ Nikol'ski{\u\i}, Translated from the 1987 Russian original, Revised by the authors.

\end{thebibliography}
\end{document}